\documentclass[final,english]{ourlematema}

\usepackage{xcolor}
\usepackage{tikz}
\usepackage{hyperref}

\usepackage{bbm}
\usepackage{algorithm}
\usepackage{algpseudocode}

\newcommand{\RR}{\mathbb{R}}

\DeclareMathOperator{\Gr}{Gr}

\title{Hyperplane Arrangements in the Grassmannian} 

\author{Elia Mazzucchelli}
\address{%
Max Planck Institute for Physics\\
Garching bei M\"unchen, Germany\\
\email{eliam@mpp.mpg.de}
}
\author{Dmitrii Pavlov}
\address{%
Max Planck Institute for Mathematics in the Sciences\\
Leipzig, Germany\\
and\\
TU Dresden, Germany (current)\\
\email{dmitrii.pavlov@mis.mpg.de}
}
\author{Kexin Wang}
\address{Harvard Univeristy\\
Cambridge, MA, USA\\
\email{kexin\_wang@g.harvard.edu}
}

\begin{document}
\maketitle
\begin{abstract}
The Euler characteristic of a smooth very affine variety encodes the algebraic complexity of solving likelihood (or scattering) equations on this variety. We study this quantity for the Grassmannian with $d$ hyperplane sections removed. We provide a combinatorial formula, and explain how to compute this Euler characteristic in practice, both symbolically and numerically. Our particular focus is on generic hyperplane sections and on Schubert divisors. We also consider special Schubert arrangements relevant for physics. We study both the complex and the real case.
\end{abstract}
\section{Introduction}
This article studies the topological Euler characteristic of the Grassmannian $\Gr_{\mathbb{K}}(k,n)$ over a field $\mathbb{K}$ with $d$ hyperplane sections removed, where $\mathbb{K} = \mathbb{C}$ or $\mathbb{R}$. When $d\geq k(n-k)$ and the hyperplanes removed form an essential arrangement, this is a smooth very affine variety. A special case of interest is the situation in which the hyperplanes we remove are Schubert. The motivation for our research comes from algebraic statistics and theoretical particle physics. 

In algebraic statistics, a discrete statistical model is the intersection of a projective variety $Y \subseteq \mathbb{CP}^N$ with the simplex $\Delta_N \subset \mathbb{CP}^N$ of real positive points. The problem of \emph{maximum likelihood estimation} consists in maximizing the \emph{log-likelihood function}
\begin{equation}
\small{L(p;s) = s_0 \cdot \log(p_0) + s_1 \cdot \log(p_1) + \ldots + s_N \cdot \log(p_N) - \left(\sum_{i=0}^N s_i\right)\cdot \log\left(\sum_{i=0}^N p_i\right)}
\end{equation}
over the points $p=(p_0:\ldots :p_N)\in Y\cap \Delta_N$ given a data vector $s = (s_0,\ldots,s_N)$. The algebraic complexity of this problem is measured by the number of complex critical points of $L$ on $Y$. This number is constant for generic~$s$, and is called the $\emph{maximum likelihood (ML) degree}$ of $Y$. Note that the critical point equations for $L$ are called \emph{likelihood equations}. 

In particle physics, the likelihood function is called the \emph{scattering potential}, and one notable instance of likelihood equations are the  \emph{Cachazo--He--Yuan (CHY) scattering equations} \cite{Cachazo_2019}. These are the saddle-point equations of the Koba--Nielsen factor in the tree level $n$-point open superstring amplitudes in the ``Gross--Mende'' limit. More generally, the CHY formalism provides a way to compute scattering amplitudes for massless particles \cite{cachazo2014c}. The connection between algebraic statistics and particle physics was observed in \cite{sturmfels2021likelihood}. 

For a smooth very affine variety $X$, the ML degree is the absolute value of the topological Euler characteristic $\chi(X)$ \cite{huh2013maximum}.
When $X$ is the very affine variety obtained by removing the hyperplane arrangement $\mathcal{H} = \{H_1,\ldots,H_d\}$ from the projective space $\mathbb{KP}^N$, the Euler characteristic has a clear combinatorial meaning.
Over $\mathbb{C}$ it counts the number of bounded regions of the corresponding real affine arrangement \cite[Theorem 1.2.1]{varchenko1995critical}, and over $\mathbb{R}$ this is the total number of regions of the arrangement.
In \cite{zaslavsky1975facing}, Zaslavsky provides formulae to compute these numbers. 
In algebraic statistics, projective hyperplane arrangements correspond to \emph{linear models} \cite[Definition 7.3.3]{sullivant2018algebraic}. 

In this article, we take a step further and ask for the Euler characteristic of the very affine variety obtained by removing an arrangement of $d$ hyperplanes from the Grassmannian $\Gr_{\mathbb{K}}(k,n)$ of $k$-planes in an $n$-dimensional vector space over a field $\mathbb{K}$.
From the point of view of algebraic geometry, this is a natural generalization of the case of projective hyperplane arrangements, since the projective space $\mathbb{KP}^{n-1}$ is itself the simplest Grassmannian $\Gr_{\mathbb{K}}(1,n)$. 
It is also inspired by physics. In the context of CHY equations the main variety of interest is the moduli space $\mathcal{M}_{0,n}$ of configurations of $n$ marked points on $\mathbb{CP}^1$. 
This is the quotient of $\Gr_{\mathbb{C}}(2,n)$ by the scaling action of the algebraic torus $(\mathbb{C}^*)^n$. 
This has been generalized to the so called $(k,n)$ \textit{CEGM formalism} \cite{Cachazo_2019}, where one considers scattering equations on the configuration space $X(k,n) = \Gr_{\mathbb{C}}(k,n) / (\mathbb{C}^*)^n$  of $n$ points on $\mathbb{CP}^{k-1}$. Since $X(k,n)$ is naturally connected to the $\Gr(k,n)$, it might also be interesting to study scattering equations on the Grassmannian itself.
We point out that in the $(k,n)$ CEGM formalism one considers a very special Schubert arrangement, namely that of $\binom{n}{k}$ hyperplanes, each defined by the vanishing of a single Pl\"ucker coordinate. An algebraic statistics perspective on $\Gr_{\mathbb{C}}(k,n)$ and $X(k,n)$ is presented in \cite[Section 4]{devriendt2024two}.

Scattering equations are also relevant in the context of positive geometries \cite{lam2022invitation}. Here they arise in the computation of the pushforward of the canonical form under a morphism of positive geometries \cite{Arkani_Hamed_2017}. This again has connections to physics, where scattering equations provide a diffeomorphism from the positive part of $\mathcal{M}_{0,n}$ to the ABHY associahedron. Thus, the canonical form of the
latter, which encodes scattering amplitudes for the bi-adjoint $\phi^3$ theory, is given by the pushforward of the former, which is computed by summing over the saddle points of the string integral \cite{Arkani_Hamed_2017}. More generally, the pushforward from the positive part of $X(k,n)$ yields amplitudes for a certain generalized bi-adjoint $\phi^3$ theory \cite{Cachazo_2019}. This pushforward procedure has also been studied in the setting of stringy canonical forms \cite{Arkani_Hamed_2021}, where the setup of Grassmannian string integrals is somewhat similar to ours. Note that Grassmannian string integrals for $k=4$ may be related to possible non-perturbative geometries for $\mathcal{N} = 4$ super Yang--Mills (SYM) amplitudes. Moreover, four-dimensional scattering equations are related to $\Gr(k,n)$ \cite{Arkani_Hamed_2011}, where $k$ encodes the helicities of the particles. Here scattering equations provide a first step towards finding a stringy canonical form for the $\mathcal{N}=4$ SYM theory.

In this article, we consider the cases $\mathbb{K}= \mathbb{C}$ and $\mathbb{K} = \mathbb{R}$.
When the ground field is that of complex numbers, we are interested in the situation when the arrangement of $d$ hyperplanes we are removing from $\Gr_{\mathbb{C}}(k,n)$ is generic. 
That is, any intersection of $i$ hyperplanes with the Grassmannian has codimension $i$ if $i\leq k(n-k)$ and is empty otherwise, and each hyperplane intersects the Grassmannian transversally. Of special interest is the situation in which the hyperplanes we remove are Schubert (Definition \ref{def:schubert}). 
On the one hand, as mentioned above, this situation is of more relevance to physics. On the other hand, Schubert sections of the Grassmannian are in general not smooth. The simplest example of this phenomenon is that of a single Schubert divisor in $\Gr_{\mathbb{C}}(2,4)$. 
It has exactly one singular point, given by the line that defines the divisor.
The non-smoothness of Schubert divisors presents additional difficulties for computing the Euler characteristic.
In the cases considered, we give formulae for $\chi(\Gr_{\mathbb{C}}(k,n)\setminus \mathcal{H})$ in terms of the Euler characteristics $\chi_{k,n}(d)$ of intersections of $d$ generic (general or Schubert) hyperplanes in $\Gr_{\mathbb{C}}(k,n)$, and show how to compute $\chi_{k,n}(d)$ symbolically (by performing computations in the Chow ring of $\Gr_{\mathbb{C}}(k,n)$) and numerically (by solving the corresponding likelihood equations). 
The results of our computations are supplemented by the \texttt{Julia} code available at \cite{mathrepo}.
We also consider two special Schubert arrangements in $\Gr_{\mathbb{C}}(2,4)$. The first one is associated to lines in $\mathbb{CP}^3$ forming a cycle. This is motivated by the $k=m=2$ tree amplituhedron \cite{ranestad2024adjoints}, whose boundary divisors form such a Schubert arrangement. The second one consists of the six coordinate hyperplanes, and is motivated by the $(2,4)$ CEGM formalism. 

In the real case one cannot speak of generic arrangements.
Instead, we present a numerical algorithm based on \cite{cummings2024smooth} to compute $\chi(\Gr_{\mathbb{R}}(k,n)\setminus \mathcal{H})$ and the number of regions in $\Gr_\RR(k,n)\setminus \mathcal{H}$ for a given arrangement. The algorithm relies on Morse theory. We also report examples with interesting topological behavior, underscoring the fact that hyperplane arrangements inside the real Grassmannian can behave differently from those in the real projective space.

The article is organized as follows. 
In Section \ref{sec:2}, we introduce the necessary notions and formulate our problem. 
In Section \ref{sec:3}, we present our main result, Theorem \ref{thm:main}, giving formulae for the Euler characteristic of $\Gr_{\mathbb{K}}(k,n)\setminus \mathcal{H}$ in the complex and real cases.
Section \ref{sec:4} is devoted to computing the Euler characteristic of $\Gr_{\mathbb{C}}(k,n)\setminus \mathcal{H}$ in practice. Here we consider generic hyperplane arrangements and Schubert arrangements. 
Finally, Section \ref{sec:5} features computational aspects in the real case.
 
\section{Preliminaries} \label{sec:2}

\begin{dfn}
    The \emph{Grassmannian} $\Gr_{\mathbb{K}}(k,n)$ is the variety parametrizing $k$-dimensional linear subspaces of an $n$-dimensional linear space over a field $\mathbb{K}$. Equivalently, it parametrizes $(k-1)$-dimensional subspaces of $\mathbb{KP}^{n-1}$. 
\end{dfn}
The Grassmannian $\Gr_{\mathbb{K}}(k,n)$ is a projective variety of dimension $k(n-k)$; it can be realized as a subvariety of the projective space $\mathbb{KP}^{\binom{n}{k}-1}$ via the \emph{Pl\"ucker embedding}. This embedding sends a $k\times n$-matrix whose rows span a linear subspace to the vector of its maximal minors, called the \emph{Pl\"ucker coordinates}. 
For more details on the Grassmannian, see e.g.~\cite[Chapter~5]{michalek2021invitation}.

\begin{dfn} \label{def:schubert}
    Consider a point $q \in \Gr_{\mathbb{K}}(n-k,n)$. It represents an $(n-k)$-dimensional subspace $V_q$ of $\mathbb{K}^n$. A generic subspace of dimension $k$ intersects $V_q$ trivially, only at the origin. All the $k$-dimensional subspaces that intersect $V_q$ non-generically, i.e.\ in a positive-dimensional subspace, form a subvariety of codimension one in $\Gr_{\mathbb{K}}(k,n)$. This subvariety is called the \emph{Schubert divisor} corresponding to $q$. 
\end{dfn}
Geometrically, a Schubert divisor is a section of $\Gr_{\mathbb{K}}(k,n)$ by a hyperplane in $\mathbb{KP}^{\binom{n}{k}-1}$. Concretely, a subspace $V_p$ corresponding to a point $p\in \Gr_{\mathbb{K}}(k,n)$ meets $V_q$ in positive dimension if and only if the following determinant vanishes: 
\begin{equation}\label{eq:laplace_schubert_hyperplane}
\det \begin{bmatrix}
    P\\
    Q
\end{bmatrix} = 0 \ .
\end{equation}
Here $P$ is a $k\times n$-matrix representing $V_p$, and $Q$ is an $(n-k)\times n$-matrix representing $V_q$. 
This relation defines a hyperplane $H_q$ in $\mathbb{KP}^{\binom{n}{k}-1}$, and any hyperplane arising in such a way is called a \emph{Schubert hyperplane}. The corresponding Schubert divisor is then $\Gr_{\mathbb{K}}(k,n)\cap H_q$. 


\begin{dfn}
    A \emph{projective hyperplane arrangement} over $\mathbb{K}$ is a finite collection of hyperplanes $\mathcal{H} = \{H_1,\ldots,H_d\}$ in $\mathbb{KP}^N$. 
We say that $\mathcal{H}$ is a \emph{Schubert arrangement} in $\mathbb{KP}^{\binom{n}{k}-1}$ if each $H_i$ is a Schubert hyperplane. 
\end{dfn}

Consider the scattering potential
\begin{equation}\label{scattering_potential}
    L(p;s) = \sum_{i=1}^d s_i \log(\det M_i(p)) \ .
\end{equation}
Here $M_i(p)$ are $n\times n$-matrices of the form 
$
\begin{bmatrix}
P
    \\
    Q_i
\end{bmatrix},
$
where $Q_i$ are fixed $(n-k)\times n$-matrices, and $p$ is the vector of Pl\"ucker coordinates of a $k\times n$-matrix $P$, which are treated as variables. The number of complex critical points of this potential is the ML-degree of $\Gr_{\mathbb{C}}(k,n)\setminus \mathcal{H}$, where $\mathcal{H}$ is the arrangement consisting of $d$ Schubert hyperplanes corresponding to $Q_i$ for $i=1,\ldots, d$. Note that when $d \geq k(n-k)$ then $\Gr_{\mathbb{C}}(k,n)\setminus \mathcal{H}$ is a smooth very affine variety (i.e.\ a closed subvariety of the algebraic torus $(\mathbb{C}^*)^N$ for some $N$). Thus, by \cite[Theorem 1]{huh2013maximum} its ML-degree is equal to the absolute value of its topological Euler characteristic $\chi(\Gr_{\mathbb{C}}(k,n)\setminus\mathcal{H})$. 
One of our objects of interest is therefore $\chi(\Gr_{\mathbb{K}}(k,n)\setminus\mathcal{H})$, where $\mathcal{H}$ is a Schubert hyperplane arrangement in $\mathbb{KP}^{\binom{n}{k}-1}$, and $\mathbb{K}$ is either $\mathbb{C}$ or $\mathbb{R}$. We will also study $\chi(\Gr_{\mathbb{K}}(k,n)\setminus\mathcal{H})$ for generic arrangements $\mathcal{H}$ of general hyperplanes, i.e. generic hyperplanes in the ambient projective space.

 In the following we will make use of some elements of Schubert calculus over $\mathbb{C}$. We refer the reader to \cite[Chapter 4]{Eisenbud20163264AA} for an introduction to the topic. Another useful concept is that of partially ordered sets (posets). Here we refer the reader to \cite[Lecture 1]{Stanley2007AnIT}.
\begin{dfn}
     The \emph{Möbius function} of a locally finite poset $(P,\leq)$ is a function $\mu \colon \text{Int}(P) \rightarrow \mathbb{Z}$ such that $\mu(x,x)=1$ and $\sum_{z \in [x,y]} \mu(x,z)=0$ for every $x,y \in P$. Note that we write $\mu(x,y):=\mu([x,y])$ and $\mu(x):= \mu([\hat{0},x])$, where $\hat{0}$ is the minimal element of $P$.
 \end{dfn}
 \begin{exa}
    The \textit{truncated boolean algebra} of rank $r$ with $d$ atoms is the poset consisting of all subsets containing up to $r$ elements of $[d] := \{1, 2, \dots , d \}$ ordered by inclusion. The $d$ single element sets are called \textit{atoms}. The boolean algebra is a graded poset with rank function given by $\text{rk}(x)=i$ for every $x \in \binom{[d]}{i}$ and $0 \leq i \leq r$. Its Möbius function is given by $\mu(x)=(-1)^{\text{rk}(x)}$, which can be easily verified using the identity $\sum_{i=0}^{m} (-1)^{i} \binom{m}{i} = 0$ for every $m>0$.
 \end{exa}
The following result appears as \cite[Theorem 1.1]{Stanley2007AnIT}.
 \begin{prop}[Möbius Inversion Formula]\label{Möb_inversion}
 Let $(P,\leq)$ be a locally finite poset with Möbius function $\mu$ and let $f,g \colon P \rightarrow \mathbb{K}$ be arbitrary functions. The following are equivalent:
 \begin{enumerate}
     \item $f(x) = \sum_{y \geq x} g(y) \quad \forall \, x \in P$ \,,
     \item $g(x) = \sum_{y \geq x} \mu(x,y)f(y) \quad \forall \, x \in P \,.$
 \end{enumerate}
 \end{prop}

\begin{dfn}
Consider the Grassmannian $\Gr_{\mathbb{K}}(k,n) \subset \mathbb{KP}^{\binom{n}{k}-1}$ and
let $\mathcal{H} = \{H_1,\ldots,H_d\}$ be a collection of $d$ hyperplanes in $\mathbb{KP}^{\binom{n}{k}-1}$. 
There is a natural poset $(\mathcal{L}(\mathcal{H}),\leq)$ associated to the hyperplane arrangement $\mathcal{H}$ in the Grassmannian which we will call the \emph{intersection poset}. Its elements are the non-empty intersections of hyperplanes inside the Grassmannian $\Gr_{\mathbb{K}}(k,n)$ (that is, elements of the form $H_{i_1}\cap\ldots\cap H_{i_l} \cap \Gr_{\mathbb{K}}(k,n)$) ordered by $x \leq y \iff y \subseteq x$ for $x,y \in \mathcal{L}(\mathcal{H})$. The intersection poset is graded of rank at most $d$ and the rank of an element is given by the minimal number of hyperplanes whose intersection with the Grassmannian yields that element.    
\end{dfn}

For $\mathcal{H}$ and $\mathcal{L}(\mathcal{H})$ as above, and for any $x \in \mathcal{L}(\mathcal{H})$, we define the following subarrangements:
\begin{equation}
\mathcal{H}_x := \{H \in \mathcal{H} : x \subseteq H\} \quad \text{and} \quad \mathcal{H}^x := \{x \cap H \neq \emptyset :  H \in \mathcal{H} \setminus \mathcal{H}_x \} \ .
\end{equation}
such that
\begin{equation}
 \mathcal{L}(\mathcal{H}_x) \cong \{y \in \mathcal{L}(\mathcal{H}) : y \leq x \} \quad \text{and} \quad  \mathcal{L}(\mathcal{H}^x) \cong \{y \in \mathcal{L}(\mathcal{H}) : x  \leq y  \}  \ .
\end{equation}

\section{Main Result}\label{sec:3}

We are now ready to state our main result.

\begin{thm} \label{thm:main}
    Let $\mathcal{H}$ be a hyperplane arrangement in $\mathbb{KP}^{\binom{n}{k}-1}$ and $\mathcal{L}(\mathcal{H})$ the intersection poset of the arrangement inside $\mathrm{Gr}_{\mathbb{C}}(k,n)$. The following equality holds for the Euler characteristic:
    \begin{equation}\label{eq:generalcount}
        \chi (\Gr_{\mathbb{K}}(k,n)\setminus \mathcal{H}) = \sum_{y\in\mathcal{L}(\mathcal{H})}\chi(y)\mu(y) \ ,
    \end{equation}
    where $\mu$ is the Möbius function of $\mathcal{L}(\mathcal{H})$.
    If $\mathbb{K} = \mathbb{C}$ and $\mathcal{H}$ is a generic arrangement of $d$ hyperplanes in $\Gr_{\mathbb{K}}(k,n)$, then its Euler characteristic is equal to $P_{k,n}(d)$, where the polynomial $P_{k,n}$ is defined as follows:
    \begin{equation} \label{eq:genericcount}
        P_{k,n}(d) = \sum_{i=0}^{k(n-k)} (-1)^i \chi_{k,n}(i) \binom{d}{i} \in \mathbb{Q}[d] \ .
    \end{equation}
    Here $\chi_{k,n}(i)$ is the Euler characteristic of the intersection of $i$ generic hyperplanes with $\Gr_{\mathbb{C}}(k,n)$. 
\end{thm}

\begin{proof}
    Our proof strategy is inspired by the proof of Zaslavsky's region counts \cite{zaslavsky1975facing} from \cite[Theorem 2.5]{Stanley2007AnIT}. By the inclusion-exclusion principle we have
    \begin{equation}
        x = \bigsqcup_{y \geq x} y \setminus \mathcal{H}^{y} \quad \forall \, x \in \mathcal{L}(\mathcal{H}) \ ,
    \end{equation}
    and by additivity of the Euler characteristic 
    \begin{equation}
        \chi(x) = \sum_{y \geq x}  \chi(y \setminus \mathcal{H}^{y}) \quad  \forall \, x \in \mathcal{L}(\mathcal{H}) \ .
    \end{equation} 
    The Möbius inversion formula (Proposition \ref{Möb_inversion}) yields
    \begin{equation}
        \chi(x \setminus \mathcal{H}^{x}) = \sum_{y \in \mathcal{L}(\mathcal{H}) \, : \, y \geq x} \chi(y)  \mu(x,y) \ .
    \end{equation}
    Setting $x=\Gr_{\mathbb{K}}(k,n)$ we obtain
    \begin{equation}\label{regions}
         \chi(\Gr_{\mathbb{K}}(k,n)\setminus \mathcal{H}) = \sum_{y \in \mathcal{L}(\mathcal{H}) } \chi(y)  \mu(y) \ .
    \end{equation}
    
    If $\mathbb{K}= \mathbb{C}$ and the arrangement $\mathcal{H}$ is generic, the intersection poset $\mathcal{L}(\mathcal{H})$ is a truncated boolean algebra, therefore its Möbius function is given by $\mu(y) = (-1)^{\mathrm{rank} (y)}$.  
    At fixed rank $i$, there are $\binom{d}{i}$ elements in $\mathcal{L}(\mathcal{H})$, corresponding to the possible ways of choosing $i$ divisors in $\mathcal{H}$ to be intersected. Over $\mathbb{C}$, by the genericity assumption, such intersections all have the same dimension $k(n-k)-i$ and the same Euler characteristic $\chi_{k,n}(i)$. By substituting this into \eqref{eq:generalcount}, we get the formula \eqref{eq:genericcount}. 
\end{proof}

\begin{rem}
    For $k=1$, we recover the setting of hyperplane arrangements in $\mathbb{KP}^{n-1}$. 
    For a generic real arrangement $\mathcal{H}$, all the elements of $\mathcal{L}(\mathcal{H})$ in codimension $i$ are isomorphic to $\mathbb{RP}^{n-1-i}$ and have the Euler characteristic equal to $(-1)^{n-1-i}$. Equation \eqref{eq:generalcount} then recovers Zaslavsky's region count 
    $$1 + (d-1) + \binom{d-1}{2} + \binom{d-1}{3} + \dots + \binom{d-1}{n} \ .$$
\end{rem}

We say that $d$ linear subspaces $V_j$ of dimension $k$ in $\mathbb{C}^{n}$ are \emph{in general position}, if any intersection of them has the minimal possible dimension, i.e.\ for every $j_1, \dots,j_i$ pairwise distinct,
\begin{equation}
    \dim(V_{j_1} \cap \dots \cap V_{j_i}) = \max(ik-(i-1)n\, ,0) \quad \forall 1 \leq i \leq d \, .
\end{equation}
In what follows, when we speak of $d$ generic Schubert divisors in $\Gr_{\mathbb{C}}(k,n)$, we mean that the $d$ many $(n-k)$-dimensional linear subspaces of $\mathbb{C}^{n}$ defining these divisors are in general position.

By replacing a generic hyperplane arrangement with a generic Schubert arrangement and arguing in a similar way, we obtain the following corollary.
\begin{cor}[Schubert arrangements]
    Let $\mathcal{H}$ be a generic Schubert arrangement in $\Gr_{\mathbb{C}}(k,n)$ with $d$ divisors. Then 
    \begin{equation}\label{Schubert:count}
        P_{k,n}^{S}(d) = \chi(\Gr_{\mathbb{C}}(k,n)\setminus \mathcal{H}) = \sum_{i=0}^{k(n-k)} (-1)^i \chi^S_{k,n}(i) \binom{d}{i} \in \mathbb{Q}[d] \ ,
    \end{equation}
    where $\chi^S_{k,n}(i)$ is the Euler characteristic of the intersection of $i$ generic Schubert divisors in $\Gr_{\mathbb{C}}(k,n)$.
\end{cor}

\section{The complex case} \label{sec:4}

In this  section, we concentrate on studying the quantities $\chi_{k,n}(i)$ and $\chi^S_{k,n}(i)$.

\subsection{Generic hyperplane arrangements in the Plücker space} \label{sec:4:1}

We start by considering the case of an arrangement $\mathcal{H}$ of $d$ generic hyperplanes in $\mathbb{CP}^{\binom{n}{k}-1}$. By Bertini's theorem, all the elements in the intersection poset $\mathcal{L}(\mathcal{H})$ are smooth. If $\text{dim}(\Gr_\mathbb{C}(k,n)) \geq 2$, which we may assume w.l.o.g., the intersections are connected and irreducible. 
In this setting, the Euler characteristic of the intersection can be computed using the Chern-Schwartz-MacPherson (CSM) class of $\Gr_\mathbb{C}(k,n)$, see \cite{Aluffi_2013}.
More generally, for a complex projective variety $X \subset \mathbb{CP}^{N}$ and generic hyperplane sections $H_i$, the Euler characteristic $\chi(X \cap H_1 \cap \dots \cap H_r)$ is completely determined by the CSM class of $X$ (denoted by $c_{{\rm SM}}(\mathbbm{1}_X)$), which is an element in the Chow group of $X$ or of the ambient projective space. If $X$ is smooth, as is the case of $X=\Gr_{\mathbb{C}}(k,n)$, the CSM class of $X$ is determined by the Chern class of the tangent bundle of $X$, denoted by $c(TX)$.
More precisely, if $X$ is nonsingular and complete, then $c_{{\rm SM}}(\mathbbm{1}_X) = c(TX) \cap [X]$, where $[X]$ is the fundamental class of $X$ in the appropriate Chow group.

We briefly summarize the results of \cite{Aluffi_2013} which are relevant for our discussion. In the above setting, we consider the topological invariant
\begin{equation}\label{chi_polynomial}
        \chi_X(t) := \sum_{r \geq 0} (-1)^r  \chi(X \cap H_{1} \cap \dots \cap H_{r}) \, t^{r} \ ,
\end{equation}
and the polynomial 
\begin{equation}
    \gamma_X(t) := \sum_{r \geq 0} \left(\int h_1^r \cdot c_{{\rm SM}}(\mathbbm{1}_X) \right) \, t^r = \sum_{r \geq 0} c_r \, t^{r} \ ,
\end{equation}
where $h_1$ is the hyperplane class in the Chow group of $\mathbb{CP}^{N}$. That is, $\gamma_X(t)$ is obtained from $ c_{{\rm SM}}(\mathbbm{1}_X) = \sum_{r \geq 0} c_r \, [\mathbb{CP}^r] $ by replacing $[\mathbb{CP}^r]$ with $t^r$. 
Consider the following transformation of polynomials
\begin{equation}\label{I_map}
    p(t) \mapsto \mathcal{I}(p)(t) := \frac{t \, p(-t-1)+p(0)}{t+1} \ ,
\end{equation}
so that the effect of $\mathcal{I}$ is to perform a sign-reversing symmetry around $t = -1/2$ of the non-constant part of $p$. \cite[Theorem 1.1]{Aluffi_2013} states that
\begin{equation}\label{Aluffi}
    \gamma_X = \mathcal{I}(\chi_X) \quad \text{and} \quad \chi_X = \mathcal{I}(\gamma_{X}) \ .
\end{equation}

This result allows to compute $\chi_{k,n}(i)$ from the Chern class of $\Gr_{\mathbb{C}}(k,n)$. Note that the computation of the latter is relatively involved, but one can compute it algorithmically, e.g. with the \texttt{Macaulay2} package \texttt{CharacteristicClasses} \cite{M2CC}, applying the function \texttt{Chern} to the ideal of the Grassmannian. It is important to notice that the output of this command gives a polynomial in the Chow ring of the ambient projective space, and hence, in order to apply \eqref{Aluffi}, one has to perform a transformation $h_1^{j} \mapsto h_1^{N-j} =: t^{N-j}$, where $N$ is the dimension of the ambient projective space. 

\begin{exa}
    Let $k=2$ and $n=4$. With \texttt{Macaulay2} we compute 
    \begin{equation}
    \gamma_{\Gr_{\mathbb{C}}(2,4)}(t) = 6 + 12 t + 14 t^{2} + 8 t^{3} + 2 t^4 \in \mathbb{Z}[t]/(t^6)  \ .
\end{equation}
    By \eqref{Aluffi} we have
    \begin{equation}\label{eq:polyeuler}
       \chi_{\Gr_{\mathbb{C}}(2,4)}(t)= \mathcal{I}(\gamma_{\Gr_{\mathbb{C}}(2,4)})(t) = 6 - 4 t + 4 t^2 - 2 t^3 + 2 t^4 \ ,
    \end{equation}
    from which we read off the Euler characteristics 
    \begin{equation} 
        \chi_{2,4}(0) = 6 \ , \quad \chi_{2,4}(1) = 4 \ , \quad \chi_{2,4}(2) = 4 \ , \quad \chi_{2,4}(3) = 2 \ , \quad \chi_{2,4}(4) = 2 \ . 
    \end{equation}
    Theorem \ref{thm:main} yields
    \begin{equation}
        P_{2,4}(d) = 6 -4d + 4\binom{d}{2} - 2\binom{d}{3} + 2\binom{d}{4}= \frac{1}{12} (72 - 86 d + 47 d^2 - 10 d^3 + d^4) \ .
    \end{equation}
    This is exactly the polynomial \eqref{eq:polyeuler} with $t^i$ replaced by $\binom{d}{i}$. 
    \begin{table}[h!] 
    \centering
    \begin{tabular}{|c|c|c|c|c|c|c|c|c|} 
    \hline
    $d$    & 4 & 5 & 6 & 7 & 8 & 9 & 10 & 11 \\
    \hline 
    $P_{2,4}(d)$  & 8 & 16 & 32 & 62 & 114 & 198 & 326 & 512   \\
    \hline  
    \end{tabular}
    \caption{The ML-degree of the complement of $d$ generic hyperplane sections in $\Gr_{\mathbb{C}}(2,4)$.}
    \label{table:ML_24}
    \end{table}
\end{exa}

\begin{exa}
    Consider $k=2$ and $n=5$. With \texttt{Macaulay2} we compute
    \begin{equation}
     \gamma_{\Gr_{\mathbb{C}}(2,5)}(t) = 10 + 30 t + 60 t^2 + 75 t^3 + 57 t^4 + 25 t^5 + 5 t^6 \in \mathbb{Z}[t]/(t^{10})  \ .
    \end{equation}
    By (\ref{Aluffi}) we have
    \begin{equation}
        \chi_{\Gr(2,5)}(t) = 10 - 8 t + 6 t^2 - 4 t^3 + 7 t^4 + 5 t^6 \ ,
    \end{equation}
    so we read off
    \begin{equation}
    \begin{aligned}
        &\chi_{2,5}(0) = 10 \ , \quad \chi_{2,5}(1) = 8 \ , \quad \chi_{2,5}(2) = 6 \ , \quad \chi_{2,5}(3) = 4 \ , \\ &\chi_{2,5}(4) = 7\ , \ \quad \chi_{2,5}(5) = 0 \ , \quad \chi_{2,5}(6) = 5 \ . 
    \end{aligned}
    \end{equation}
    The polynomial (\ref{eq:genericcount}) reads
    \begin{equation}
    \begin{aligned}
        P_{2,5}(d) &= \frac{1}{144} (1440 - 2148 d + 1456 d^2 - 573 d^3 + 127 d^4 - 15 d^5 + d^6) \\
        &= \frac{1}{144} (d-2) (d-3) (240 - 158 d + 71 d^2 - 10 d^3 + d^4) \ .
    \end{aligned}
    \end{equation}
    \begin{table}[h!] 
    \centering
    \begin{tabular}{|c|c|c|c|c|c|c|c|c|} 
    \hline
    $d$  & 4 & 5 & 6 & 7 & 8 & 9 & 10 & 11 \\
    \hline 
    $P_{2,5}(d)$ & 5 & 25 & 82 & 220 & 520 & 1120 & 2240 & 4212   \\
    \hline  
    \end{tabular}
    \caption{The ML-degree of the complement of $d$ generic hyperplane sections in $\Gr_{\mathbb{C}}(2,5)$.}
    \label{table:ML_25}
    \end{table}
\end{exa}

\subsection{Generic Schubert arrangements}

When we move from generic hyperplane arrangements to arrangements of Schubert divisors in $\Gr_{\mathbb{C}}(k,n)$, the main difference is that intersections of Schubert divisors need not be smooth. See Example \ref{ex:adhoc} for an illustration. This presents a new computational challenge, since for non-smooth varieties the Chern class no longer encodes the Euler characteristic. As noted in Section \ref{sec:4:1}, one has to turn to a finer concept, that of CSM classes. CSM classes can still be computed algorithmically \cite{helmer2016algorithms}. However, this computation is in general more involved than that of Chern classes. For instance, we failed to compute CSM classes of Schubert sections of $\Gr_{\mathbb{C}}(2,5)$ in \texttt{Macaulay2}, while the computation of the corresponding Chern classes was almost instantaneous. 

In view of these computational difficulties, it is desirable to avoid computing CSM classes whenever possible in the process of computing Euler characteristics. In this section, we offer two ways of doing so. On the one hand, we show that for a sufficiently large number of Schubert divisors, the intersection \emph{is in fact smooth}, and one can thus infer information about its Euler characteristic from the Chern class, as shown in the previous section. On the other hand, we offer a \emph{recursive approach} to computing Euler characteristics of intersections of a sufficiently small number of Schubert divisors.

Before that, we present the examples of $\Gr_{\mathbb{C}}(2,4)$ and $\Gr_{\mathbb{C}}(2,5)$ where we computed the ML degree of the complement of generic Schubert arrangements by counting the number of solutions to the scattering equations, i.e.\ by finding the non-singular critical points of \eqref{scattering_potential}. For that we use the \texttt{Julia} package \texttt{HomotopyContinuation.jl} \cite{breiding2018homotopycontinuation}. Our code is available at~\cite{mathrepo}. The numerical results of Examples \ref{ex:24num} and \ref{ex:gr25} are verified theoretically in Examples \ref{ex:252}, \ref{ex:chi24} and \ref{ex:adhoc}. 
\begin{exa} \label{ex:24num}
    Let $k=2$ and $n=4$. We compute numerically
    \begin{equation}
        P_{2,4}^S(d)= \frac{1}{12}(72 - 98 d + 47 d^2 - 10 d^3 + d^4) \ ,
    \end{equation}
    which corresponds to (\ref{Schubert:count}) with 
    \begin{equation}\label{chis_24}
        \chi_{2,4}^S(0) = 6 \ , \quad \chi_{2,4}^S(1) = 5 \ , \quad \chi_{2,4}^S(2) = 4 \ , \quad \chi_{2,4}^S(3) = 2 \ , \quad \chi_{2,4}^S(4) = 2  \ .
    \end{equation}
    \begin{table}[h!] 
    \centering
    \begin{tabular}{|c|c|c|c|c|c|c|c|c|} 
    \hline
    $d$  & 4 & 5 & 6 & 7 & 8 & 9 & 10 & 11 \\
    \hline 
    $P_{2,4}^S(d)$ & 4 & 11 & 26 & 55 & 106 & 189 & 316 & 501   \\
    \hline  
    \end{tabular}
    \caption{The ML-degree of the complement of $d$ generic Schubert divisors in $\Gr_{\mathbb{C}}(2,4)$.}
    \label{table:ML_24schubert}
    \end{table}
    \begin{table}[h!] 
    \centering
    \begin{tabular}{|c|c|c|c|c|c|c|c|c|} 
    \hline
    $d$  & 4 & 5 & 6 & 7 & 8 & 9 & 10 & 11 \\
    \hline 
    $P_{2,5}^S(d)$ & 1 & 10 & 46 & 150 & 400 & 931 & 1960 & 3816   \\
    \hline 
    \end{tabular}
    \caption{The ML-degree of the complement of $d$ generic Schubert divisors in $\Gr_{\mathbb{C}}(2,5)$.}
    \label{table:ML_25schubert}
    \end{table}
\end{exa}

\begin{exa} \label{ex:gr25}
    Let $k=2$ and $n=5$. We compute numerically
    \begin{equation}
    \begin{aligned}
        P_{2,5}^S(d) &= \frac{1}{144}(1440 - 2580 d + 1816 d^2 - 645 d^3 + 127 d^4 - 15 d^5 + d^6) \\
        & = \frac{1}{144}(6 - 5 d + d^2)^2 (40 - 5 d + d^2) \ ,
    \end{aligned}
    \end{equation}
    which corresponds to (\ref{Schubert:count}) with
    \begin{equation}\label{chis_2_5}
    \begin{aligned}
           &\chi_{2,5}^S(0) = 10 \ , \quad \chi_{2,5}^S(1) = 9 \ , \quad \chi_{2,5}^S(2) = 8 \ , \quad \chi_{2,5}^S(3) = 7 \ , \\ 
           &\chi_{2,5}^S(4) = 7 \ , \quad \chi_{2,5}^S(5)= 0 \ , \quad \chi_{2,5}^S(6) = 5  \ .
    \end{aligned}
    \end{equation}
\end{exa}

We now discuss when the intersection of $d$ generic Schubert divisors is smooth and its Euler characteristic agrees with that of a generic section of the Grassmannian by $d$ hyperplanes.

\begin{prop}
      Consider the intersection $\mathcal{D}$ of $d$ generic Schubert divisors in $\Gr_{\mathbb{C}}(k,n)$. If $d \geq \binom{n}{k} - k(n-k)$, then there exist $d$ hyperplane sections in $\mathbb{CP}^N$ in general position whose intersection in $\Gr_{\mathbb{C}}(k,n)$ is equal to $\mathcal{D}$. In particular, the intersection is smooth and 
     \begin{equation}
         \chi_{k,n}^{S}(d) = \chi_{k,n}(d) \qquad \forall \, d \geq \binom{n}{k} - k(n-k) \ .
     \end{equation}
\end{prop}

\begin{proof}
    The projective dual of a hyperplane in $\mathbb{CP}^{N}$ is a point in the dual projective space $(\mathbb{CP}^{N})^{*}$ and $\mathcal{D}$ is dual to an affine space $A \subset (\mathbb{CP}^{N})^{*}$ of dimension $d-1$. Since the hyperplanes are Schubert, their dual points lie on the dual Grassmannian $\Gr_{\mathbb{C}}(n-k,n) \subset (\mathbb{CP}^{N})^{*}$. Due to the assumption on $d$ we can, however, take any other $d$ points spanning $A$, which are dual to some hyperplanes $H_i \subset \mathbb{CP}^{N}$ for $i= 1, \dots, d$, and we can choose them to be in general position. Since the dual points $H_i^{*}$ span $A$, it follows that $H_1 \cap \dots \cap H_d \cap \Gr_{\mathbb{C}}(k,n) = \mathcal{D}$, which is therefore smooth by Bertini's Theorem.
\end{proof}
The numbers $\chi_{k,n}^S(d)$ for $d \geq \binom{n}{k} - k(n-k)$ can be therefore computed as in Section \ref{sec:4:1}, or by the adjunction formula (see \cite{cynk2011euler}) with the knowledge of the Chern class of the Grassmannian as well as its cohomology ring structure.

We now move on to the recursive approach to compute some of the Euler characteristics $\chi^S_{k,n}(d)$.
We start from the initial conditions for the recursion. These are given by Equation \eqref{EC_Gr} and by Lemma \ref{one_divisor}.
\begin{equation}\label{EC_Gr}
    \chi_{k,n}^S(0) = \chi(\Gr_{\mathbb{C}}(k,n)) = \binom{n}{k} \ .
\end{equation}
That $\chi_{k,n}^S(0)$ satisfies the same recursive relation as the binomial coefficient follows by fixing a one-dimensional subspace $L \subset \mathbb{C}^{n}$ and partitioning $\Gr_{\mathbb{C}}(k,n)$ into two sets, characterized by those elements containing $L$ and those that do not. The map sending a subspace to its quotient space with $L$, maps the former subset isomorphically to $\Gr_{\mathbb{C}}(k-1,n-1)$, while the latter becomes a rank $k$ vector bundle over $\Gr_{\mathbb{C}}(k,n-1)$. Then, the recursion relation for $\chi_{k,n}^S(0)$ follows from standard properties of the Euler characteristic.

\begin{lemma}\label{one_divisor}
    The Euler characteristic of one Schubert divisor is given by
    \begin{equation}\label{EC_one_div}
        \chi_{k,n}^S(1) = \chi_{k,n}^S(0)-\chi(\mathbb{C}^{k(n-k)}) = \binom{n}{k} - 1 \ .
    \end{equation}
\end{lemma}

\begin{proof}
    Any Schubert divisor, being a closed Schubert cell of codimension one, inherits the CW-complex structure of the Grassmannian with the (open) top-cell removed, i.e.\ it has one (even-dimensional) cell fewer than $\Gr_{\mathbb{C}}(k,n)$. 
\end{proof}

In the following we assume that ${0<k<n}$.
\begin{prop}\label{recursion_prop}
    If $n> d k$ for $d \in \mathbb{N}$, then the following recursion relation holds:
    \begin{equation}
        \chi_{k,n}^S(d) = \chi_{k-1,n-1}^S(0) + \chi_{k,n-1}^S(d) = \binom{n-1}{k-1} + \chi_{k,n-1}^S(d)  \ .
    \end{equation}
\end{prop}

\begin{proof}
    The inequality $n> d k \iff d(n-k)>(d-1)n$ implies that the intersection of the $d$ subspaces in $ \Gr_{\mathbb{C}}(n-k,n)$ associated to the $d$ generic Schubert divisors contains a one-dimensional subspace $L$. Then, points in the intersection of the divisors can be partitioned into two sets, depending on if they contain $L$ or not. Considering the projection map on the orthogonal complement of $L$, one sees that the former set is isomorphic to $\Gr_{\mathbb{C}}(k-1,n-1)$, while the latter to a rank $k$ $\mathbb{C}$-vector bundle over the intersection of $d$ generic Schubert divisors in $\Gr_{\mathbb{C}}(k,n-1)$. 
\end{proof}

\begin{exa} \label{ex:252}
    Let $k=2$, $n=5$ and consider $d=2$. By Proposition \ref{recursion_prop} we have 
    \begin{equation}
        \chi_{2,5}^S(2) = \chi_{1,4}^S(0) + \chi_{2,4}^S(2) = \binom{4}{1} + 4 = 8 \ ,
    \end{equation}
    where we relied on the numerical value $\chi_{2,4}^S(2)=4$ from (\ref{chis_24}), which is also computed in Example \ref{ex:chi24}. This agrees with \eqref{chis_2_5} and with the following result.
\end{exa}

\begin{prop}\label{2_div_recurs}
    For the intersection of two generic Schubert divisors we have
    \begin{equation}
        \chi_{k,n}^S(2) = \chi_{k-1,n-1}^S(1) + \chi_{k,n-1}^S(1) = \binom{n}{k} - 2 \ .
    \end{equation}
\end{prop}
\begin{proof}
    Note that $\Gr_{\mathbb{K}}(k,n) \cong \Gr_{\mathbb{K}}(n-k,n)$, meaning that $\chi^S_{k,n}(d) = \chi^S_{n-k,n}(d)$. We can therefore assume w.l.o.g. that $n \leq 2k$.
    This condition tells us that the subspaces associated to the divisors do not intersect and therefore we can choose a one-dimensional space contained only in one of them. We can partition the divisors' intersection into two sets as in Proposition \ref{recursion_prop}. One is isomorphic to one Schubert divisor in $\Gr_{\mathbb{C}}(k-1,n-1)$, while the other to a fiber bundle over one divisor in $\Gr_{\mathbb{C}}(k,n-1)$ with fiber isomorphic to an affine space over $\mathbb{C}$ of dimension $k-1$. 
\end{proof}

\begin{exa} \label{ex:chi24}
    Let $k=2$, $n=4$ and consider $d=2$. Combining Proposition \ref{2_div_recurs} with (\ref{EC_Gr}) we obtain
    \begin{equation}
        \chi_{2,4}^S(2) = \chi_{2,4}^S(0) - 2 = 6-2=4,
    \end{equation}
    which agrees with the numerical value from (\ref{chis_24}).
\end{exa}

Sometimes similar techniques allow to compute the Euler characteristic directly, by analyzing the structure of the variety. This is shown in the following example.

\begin{exa} \label{ex:adhoc}
    We want to compute the Euler characteristic $\chi_{2,5}^S(3) = 7$ from first principles, therefore fully theoretically reproducing the results of Example \ref{ex:gr25}. We start by noticing that the singular locus of the intersection of three generic Schubert divisors in $\Gr_{\mathbb{C}}(2,5)$ consists of three distinct points. We can then partition the smooth locus into two subsets determined by whether an element  contains or not the one-dimensional space given by the intersection of the subspaces associated to two Schubert divisors. The first set is isomorphic to $\mathbb{CP}^{2}$ minus two points.
    The second set forms a one-dimensional affine fiber bundle with base space isomorphic to the intersection of two generic Schubert divisors in $\Gr_{\mathbb{C}}(2,4)$ minus one point. Therefore, the Euler characteristic of the smooth locus is equal to four, and hence $\chi_{2,5}^S(3) = 3 + 4 = 7$, agreeing with the numerical result in (\ref{chis_2_5}).
\end{exa}

Even though the same approach of the proof of Proposition \ref{recursion_prop} may work in other cases, as we have just seen, we now comment on its validity in general. For given positive integers $0<k<n$ there exists a unique $d_{*} \geq 1$ such that $n>d_{*}k$ and $n \leq (d_{*}+1)k$. Then, Proposition \ref{recursion_prop} holds for every $d \leq d_{*}$. For $d > d_{*}$, one can still find a one-dimensional space $L$ contained in the intersection of the first $d_{*}$ Schubert divisors. By genericity, $L$ is then not contained in the other divisors. We can then write the intersection of $d$ Schubert divisors as the disjoint union of two sets: one is the same as the intersection of $d-d^*$ generic Schubert divisors in $\Gr_{\mathbb{C}}(k-1,n-1)$, and has Euler characteristic $\chi_{k-1,n-1}^{S}(d-d^{*})$, while the second one is given by
\begin{equation}
\begin{aligned}\label{decomposition}
        \{V\in \Gr_{\mathbb{C}}(k,n) : V \cap L = \{0\} \ , \ V \cap D_{i} \neq \{0 \} \quad \text{for} \ 1 \leq i \leq d \} \ ,
\end{aligned}
\end{equation}
where $D_i \in \Gr_{\mathbb{C}}(n-k,n)$ are the subspaces associated to the divisors. The difficulty in generalizing the recursion relations is that (\ref{decomposition}) may not be a fiber bundle in general.

\subsection{Special arrangements in $\Gr_{\mathbb{C}}(2,4)$}

Motivated by the CEGM formalism and the $k=m=2$ amplituhedron, in this section we consider two non-generic Schubert arrangements in $\Gr_{\mathbb{C}}(2,4)$. One of them consists of $d$ Schubert divisors corresponding to lines in $\mathbb{CP}^3$ forming a cycle (meaning that the lines can be ordered such that two adjacent lines intersect). An instance of such a configuration is given by the four positroid hyperplanes $p_{12} = p_{14} = p_{23} = p_{34} = 0$. The other arrangement we consider is given by the hyperplanes $p_{12} = p_{13} = p_{14} = p_{23} = p_{24} = p_{34} = 0$ in~$\mathbb{CP}^5$. 

\begin{exa}
    Let $\mathcal{H}_d$ be the arrangement of $d$ Schubert divisors in $\mathbb{CP}^5$ corresponding to $d$ lines in $\mathbb{CP}^3$ forming a cycle, and $X_d = \Gr_{\mathbb{C}}(2,4) \setminus \mathcal{H}_d$. The intersection poset in this case is still a truncated boolean algebra, although the rank function is now not given by the codimension of the intersections in $\mathbb{CP}^5$. We recall $\chi(\Gr_{\mathbb{C}}(2,4)) = 6$. The Euler characteristic of a single Schubert divisor is $5$, and that of the intersection of two divisors is still always $2$. When three divisors are intersected, there are two possibilities. If the divisors correspond to three disjoint lines in $\mathbb{CP}^3$, then the Euler characteristic is equal to $2$, and in all other cases it is $3$. There are $d(d-4)+d = d(d-3)$ many ways to intersect divisors coming from disjoint lines, and therefore $\binom{d}{3} - d(d-3)$ triple intersections of other types. For the intersection of four divisors the Euler characteristic is always $2$. Now, applying Theorem \ref{thm:main}, we obtain 
    \begin{equation}
\begin{aligned}
    \chi(X_d) &= 6 - 5d + 4 \binom{d}{2} -  2 \frac{d(d-4)(d-5)}{6} - 3 d(d-3) + 2 \binom{d}{4} \\
    &= \frac{1}{12}(72 - 62d + 35d^2 - 10d^3 + d^4) \qquad \forall \, d \geq 4 \ .
\end{aligned}
\end{equation}
For $d = 3$ one can directly compute $\chi(X_3) = 0$.
\begin{table}[h!] 
\centering
\begin{tabular}{|c|c|c|c|c|c|c|c|c|c|} 
  \hline
  $d$ &4 & 5 & 6 & 7 & 8 & 9 & 10 & 11 \\
  \hline 
  $\chi(X_d)$ &0 &1 &8 &27 & 66 & 135 & 246 & 413\\
  \hline 
  
\end{tabular}
\caption{The ML-degree of $X_d$, the complement in $\Gr_{\mathbb{C}}(2,4)$ of a Schubert arrangement corresponding to $d$ lines in $\mathbb{CP}^3$ forming a cycle. These values can be compared to those in Tables \ref{table:ML_24} and \ref{table:ML_24schubert}. We also checked the data with our \texttt{Julia} code.}
\label{table:ML_Xn}
\end{table}
\end{exa}

\begin{exa}
    Consider the arrangement of the six Pl\"ucker hyperplanes in $\Gr_{\mathbb{C}}(2,4)$. Unlike in the generic case, the intersection of any $5$ divisors is now non-empty, it consists of a unique point. The Euler characteristics for the intersections of $2, 3, 4$ and $5$ divisors in this arrangement are $4, 3, 2$ and $1$ respectively. We still have $6$ and $5$ for the Euler characteristics of $\Gr_{\mathbb{C}}(2,4)$ and a single Schubert divisor. Theorem \ref{thm:main} now yields
    \begin{equation}
    6-5 \cdot 6 + 4 \binom{6}{2} - 3 \binom{6}{3} + 2 \binom{6}{4} - 1 \binom{6}{5} = 0 \ ,
    \end{equation}
    which we also verified with our code in \texttt{Julia}.
\end{exa}

\begin{exa}
    Consider as above the arrangement of the six Pl\"ucker hyperplanes in $\Gr_{\mathbb{C}}(2,4)$ and one extra Schubert divisor, whose associated line in $\mathbb{P}^3$ is disjoint from every other line coming from the Pl\"ucker hyperplanes. The analysis of the poset's structure is similar to above, with the main difference that there are three triple intersections associated to lines not intersecting each other. The Euler characteristic of the latter is equal to $2$, while all the other characteristics are as above. Moreover, out of the $\binom{7}{5}=21$ intersections of five divisors, only $18$ are nonempty. Also, all the intersections of six divisors are empty. Therefore, by Theorem \ref{thm:main} the Euler characteristic of the complement is
    \begin{equation}
    6-5 \cdot 7 + 4 \binom{7}{2} - 3 \left(\binom{7}{3}-3 \right) - 2 \cdot 3 + 2 \binom{7}{4} - 1 \cdot 18 = 4 \ ,
    \end{equation}
    which agrees with the value computed in \cite[Theorem 4.1]{devriendt2024two}.
\end{exa}

\section{The real case} \label{sec:5}

When considering a hyperplane arrangement $\mathcal{H}$ in $\mathbb{R}^n$, the Euler characteristic $\chi(\mathbb{R}^n\setminus \mathcal{H})$ has a clear combinatorial meaning. 
It counts the total number of regions (connected components) of $\mathbb{R}^n\setminus \mathcal{H}$ since each region is contractible and has Euler characteristic one. 
Moreover, the regions of $\mathbb{R}^n\setminus \mathcal{H}$ are in bijection with the maximal covectors of the underlying oriented matroid $\mathcal{M}_\mathcal{H}$~\cite[Chapter 1]{orientedmatroids}. 
The latter are also known as sign patterns, i.e.\ sign vectors realized by hyperplane equations evaluated at a point in some region.

Such combinatorial interpretation is unfortunately not immediately available for $\Gr_{\mathbb{R}}(k,n) \setminus \mathcal{H}$. 
One may ask the following natural questions.
\begin{enumerate}
\item Are all regions contractible?
\item How many regions and sign patterns are there?
\item Is the number of regions equal to the number of sign patterns?
\end{enumerate}
We will answer these questions via examples. 
We note right away that the answers to 1 and 3 are negative and this demonstrates a difference from hyperplane arrangements in $\RR^n$.
Our approach here is based on Morse Theory. This is in contrast to the techniques for computing the Euler characteristic presented in Section \ref{sec:3}, which only work over $\mathbb{C}$.

In \cite{cummings2024smooth}, the authors introduced an algorithm based on Morse Theory and solving ODEs to compute the Euler characteristic and the number of regions for $\RR^n\backslash V(f(x_1,\ldots,x_n))$ where $f(x_1,\ldots,x_n) \in \RR[x_1,\ldots,x_n]$. 
By taking $f(x_1,\ldots,x_n)$ to be a product of polynomials, the algorithm can be applied to hypersurface arrangements.
One can also obtain the Euler characteristic and sign pattern for each region.
We will present a slightly different version of \cite[Algorithm 1]{cummings2024smooth} in Algorithm \ref{alg:Hauenstein}, and show how to apply it to hyperplane arrangements in $\Gr_\RR(k,n)$. This will answer the three questions raised above.

\begin{algorithm}[htb]
\caption{Regions of $\RR^n\backslash \bigcup_{i=1}^d V(f_i)$}\label{alg:Hauenstein}
\begin{algorithmic}[1]
\renewcommand{\algorithmicrequire}{\textbf{Input:}}
\Require Equations of $d$ hypersurfaces $f_1(x_1,\ldots,x_n),\ldots,f_d(x_1,\ldots,x_n)$ in $\RR^n$

\State Choose a generic degree two polynomial $g(x_1,\ldots,x_n)$ with $g(x)>0$ for any $x\in \RR^n$ and define a rational function $r:=\frac{f_1\cdots f_d}{g^\ell}$ where $2\ell>\sum_{i=1}^d \deg f_i$.

\State Compute the critical points $\{p_1,\ldots,p_M\}$ of $-\log |r|$, check the Hessian matrices $H_{p_i}$ at $p_i$ are non-degenerate and compute their eigendecompositions. 

\State Record the index (number of negative eigenvalues of $H_{p_i}$) and unstable eigenvectors (those with negative eigenvalues) of each critical point $p_i$.

\State The Euler characteristic $\chi(\RR^n\backslash\bigcup_{i=1}^d V(f_i))$ is given by $\sum_{\ell = 0}^n (-1)^\ell \mu_\ell$, where $\mu_\ell$ is the number of index $\ell$ critical points \cite{mukherjee2015morse}.

\State Initialize a graph $G$ with vertices $\{p_1,\ldots,p_M\}$ and no edges.

\State For each index one critical point $p_{i}$ with unstable eigenvector $v$, solve the gradient ascend ODE with starting points $p_{i}+\epsilon v, p_{i}-\epsilon v$ for small $\epsilon>0$. 
Add the edges between the two critical points the ODE solutions limit to and $p_{i}$ to $G$.

\State For each index $>1$ critical point $p_{i}$ with one unstable eigenvector $v$, solve gradient ascend ODE with starting point $p_{i}+\epsilon v$ for small $\epsilon>0$. Add the edge between the critical point the ODE solution limits to and $p_{i}$ to $G$.

\State Compute the connected components of $G$. For each region, compute the sign pattern by evaluating the polynomials $(f_1,\ldots,f_d)$ at some critical point and Euler characteristic using indices of the critical points in the region.
\renewcommand{\algorithmicensure}{\textbf{Output:}}
\Ensure All realizable sign patters in $\{-,+\}^d$, and all  regions,  with their Euler characteristics.
\end{algorithmic}
\end{algorithm}

The idea of Algorithm \ref{alg:Hauenstein} is as follows. 
We define a rational Morse function $r$ that vanishes on the hypersurfaces and at infinity. 
This ensures that each region of $\RR^n\setminus\bigcup_{i=1}^d V(f_i)$ has at least one local minimum (index zero critical point) of $-\log|r|$. 
Then to figure out whether two critical points live in the same region, one starts from a critical point that is not a local maximum and does gradient ascend. 
The path will not leave the region and limit to some other critical point. 
Any critical point will eventually travel to some local maximum and any two local maxima in the same region are connected by some index one critical point via the Mountain Pass Theorem, see e.g. \cite[Theorem 3]{more2004computing}.
Hence, these path trackings give full information to partition the critical points into subsets of points belonging to the same region.
From this information one also obtains the sign pattern and the Euler characteristic of each region.

We now explain how to apply Algorithm \ref{alg:Hauenstein} to the real hyperplane arrangement $\mathcal{H}=\{H_1,\ldots,H_d\}$ in $\Gr_\RR(k,n)$, with $H_1$ being a Schubert hyperplane.
The importance of having one Schubert hyperplane is that the complement of it in $\Gr_\RR(k,n)$ has a parametrization delivering a simple isomorphism with $\RR^{k(n-k)}$.

By the discussion in Section \ref{sec:2}, one obtains the equation of a Schubert hyperplane in Pl\"ucker coordinates via the Laplace expansion in \eqref{eq:laplace_schubert_hyperplane}.
By a possible change of basis, one can assume in \eqref{eq:laplace_schubert_hyperplane} that
$$Q= \begin{bmatrix}
    0_{(n-k)\times k}, I_{n-k}
\end{bmatrix}.$$ 
The corresponding Schubert hyperplane is $p_{1,\ldots,k}=0$ and there is an isomorphism 
$$\phi: \RR^{k(n-k)} \xrightarrow{\sim} \Gr_\RR(k,n) \backslash \{p_{1,\ldots,k}=0\}, $$
given by mapping the maximal minors of 
$$
P'=\begin{bmatrix}
& & &x_{1,1}&\cdots&x_{1,n-k}\\
&I_k& & \vdots& \ddots&\vdots\\
& & &x_{k,1}&\cdots&x_{k,n-k}
\end{bmatrix}
$$
to the Pl\"ucker coordinates.
We denote the equations of $H_2,\ldots,H_d$ in variables $\{x_{1,1},x_{1,2},\ldots,x_{k,n-k}\}$ by $f_2,\ldots,f_d$ respectively. 
They have degrees at most $k$. 
The regions in $\RR^{k(n-k)}\backslash\bigcup_{i=2}^d V(f_i)$ and their Euler characteristics are the same as those for $\Gr_\RR(k,n)\backslash \mathcal{H}$ by the isomorphism $\phi$. 
Therefore, we apply Algorithm \ref{alg:Hauenstein} to $\RR^{k(n-k)}\backslash\bigcup_{i=2}^d V(f_i)$.

We conclude this section with examples that answer the questions raised at the start of this section.
We will focus on the Grassmannian $\Gr_\RR(2,4)$. We fix one Schubert hyperplane to be $\{p_{12}=0 \}$ and we can then work in $\RR^{4}\cong \Gr_\RR(2,4)\backslash \{p_{12}=0\}$ parametrized by $x_{11},x_{12},x_{21},x_{22}$.
The code is available at \cite{mathrepo} and is based on the \texttt{Julia} package \texttt{HypersurfaceRegions} \cite{breiding2024computing}.

Regions of a hyperplane arrangement in the Grassmannian are not always contractible. 
In the following example, all regions are non-contractible.
\begin{exa}
We consider the Schubert hyperplanes associated to the four $2\times 4$ matrices
$$
\begin{bmatrix}
    0&0&0&1\\
    0&0&1&1
\end{bmatrix},\begin{bmatrix}
0&1&0&1\\ 1&1&0&1
\end{bmatrix}, \begin{bmatrix}
1&0&1&1\\1&1&1&1
\end{bmatrix},
\begin{bmatrix}
0&1&1&1\\
1&0&0&1
\end{bmatrix}.
$$

They correspond to lines in $\mathbb{RP}^3$ that intersect the four lines connecting pairs of vertices in a $(2,2,2,2)$ partition of 8 vertices in a unit cube.
The Schubert hyperplane for the first matrix is $p_{12}=0$. 

The equation for the Schubert hyperplane corresponding to the second matrix is 
$$f_2=\det \begin{bmatrix}
0&1&0&1\\
1&1&0&1\\
1&0&x_{11}&x_{12}\\
0&1&x_{21}&x_{22}
\end{bmatrix}=x_{11}-(x_{11}x_{22}-x_{12}x_{21}).$$ 
Similarly, $f_3=x_{21}-x_{22}+(x_{11}x_{22}-x_{12}x_{21}),\ f_4=1+x_{11}-x_{12}-x_{21}-(x_{11}x_{22}-x_{12}x_{21})$. 

We apply Algorithm \ref{alg:Hauenstein} to $\RR^4\backslash\bigcup_{i=2}^4 V(f_i)$. There are eight regions in total and each corresponds to a distinct sign pattern.
There are seven regions with Euler characteristic 0 and each of them has the homology type of a circle. The region with
$f_2>0,\ f_3<0,\ f_4>0$ has Euler characteristic $-2$.
No region has Euler characteristic 1, so all the regions are non-contractible.
\end{exa}

It is also possible to have more than one region per sign pattern.
\begin{exa}
Consider the four Pl\"ucker hyperplanes given by the vanishing of $p_{12}, p_{14},\  p_{23}, \ p_{34}$ in $\Gr_\RR(2,4)$. 
In $\RR^{4}\cong \Gr_\RR(2,4)\setminus\{p_{12}=0\}$, $p_{14}=0,\ p_{23}=0,\ p_{34}=0$ become $x_{22}=0,\ -x_{11}=0,\ x_{11}x_{22}-x_{12}x_{21}=0$.
The output of Algorithm \ref{alg:Hauenstein} for $\RR^4\setminus V(x_{22}x_{11}(x_{11}x_{22}-x_{12}x_{21}))$ is as follows. 
There are 12 regions in total. All regions are contractible and all possible sign patterns appeared. 
Each of the sign patterns $(- - +),(+ - -),(- + -),(+ + +)$ has two regions. 
The other four sign patterns have one region each.

We consider another example with four Schubert hyperplanes.
We take one Schubert hyperplane $\{p_{12}=0\}$ and the other three are given by the matrices
$$
\begin{bmatrix}
4 & 6 & -10 & 2 \\
6 & 4 & -5 & -14
\end{bmatrix}, 
\begin{bmatrix}
1 & 4 & 7 & 14\\
26 &  0 & -11 & 1
\end{bmatrix},
\begin{bmatrix}
14 &  7 & -4 & -8\\ -6 & 7  & 2 &   7
\end{bmatrix}.
$$
We apply Algorithm \ref{alg:Hauenstein} to $\RR^4\backslash\bigcup_{i=2}^4V(f_i)$. 
There are 9 regions in total. 
All sign patterns appeared. 
The sign pattern $(- - -)$ contains two contractible regions. 
For the 7 remaining regions,
the region with sign pattern $(+, +, -)$ has Euler characteristic $-2$, the region with sign pattern $(+, +, +)$ has Euler characteristic $1$ and the other 5 regions have Euler characteristic 0.
\end{exa}

\begin{exa}
We now report the number of regions for sampling $n=4,5,6$ random Schubert hyperplanes in 100 trials.
We always fix one Schubert hyperplane to be $p_{12}=0$ and we sample the rest of them by sampling $2\times 4$ matrices with standard Gaussian entries. 

  

\begin{table}[htbp] 
\centering
\begin{tabular}{|c|c|c|c|c|c|c|c|c|c|} 
  \hline
  $\# \ \text{Schubert hyperplanes}$ & \# \text{Regions } & \text{Probability} (\%)\\
  \hline 
  $4$ &(8,9) & (82,18) \\
  \hline 
  $5$ & (16,17,18,19) & (31,28,30,11)   \\
  \hline
  $6$ & (32,33,34,35,36,& (5,11,9,11,18,\\
   &  37,38,39,40,41) & 14,11,15,3,3)\\
  \hline
\end{tabular}
\caption{Number of regions for $4,5,6$ Schubert hyperplanes in $\Gr_{\mathbb{R}}(2,4)$.}
\label{table:connected components}
\end{table}


For comparison, we also report the number of regions for sampling $n=3,4,5$ general hyperplanes and fix an additional Schubert hyperplane $p_{12}=0$ in 100 trials.
We sample the coefficients of the general hyperplanes with standard Gaussians.

\begin{table}[htbp] 
\centering
\begin{tabular}{|c|c|c|c|c|c|c|c|c|c|} 
  \hline
  $\# \ \text{General hyperplanes}$ & \# \text{Regions } & \text{Probability} (\%)\\
  \hline 
  $3$ &(8,9) & (92,8) \\
  \hline 
  $4$ & (15,16,17,18,19,20,21) & (1,42,23,21,6,4,2)   \\
  \hline
  $5$ & (28,30,32,33,34,35,36& (1,1,1,13,14,15,18,15\\
   & 37,38,39,41,42) & 9,2,2,1,1)\\
  \hline
\end{tabular}
\caption{Number of regions for $4,5,6$ hyperplanes in $\Gr_{\mathbb{R}}(2,4)$ with one hyperplane $p_{12}=0$ and the rest general.}
\label{table:connected components general}
\end{table}

\end{exa}

\section*{Acknowledgments}
The authors are grateful to Bernd Sturmfels for initiating the project and to Kristian Ranestad for useful discussions. EM also thanks Prashanth Raman, Jonah Stalknecht and Nima Arkani-Hamed for helpful discussions. EM is funded by the European Union (ERC, UNIVERSE PLUS, 101118787). Views and opinions expressed are however those of the authors only and do not necessarily reflect those of the European Union or the European Research Council Executive Agency. Neither the European Union nor the granting authority can
be held responsible for them.

\bibliography{main}

\end{document}